\documentclass[a4paper,12pt, legno]{article}
\usepackage{amsmath}
\usepackage{amsfonts}
\usepackage{amsthm}
\usepackage{mathrsfs}
\usepackage[T1]{fontenc}
\usepackage{amsmath}
\usepackage{amsfonts}
\usepackage{amsthm}
\usepackage{mathrsfs}
\usepackage{amssymb}
\usepackage{pst-node}
\usepackage{tikz-cd}
\newtheorem{thm}{Theorem}
\newtheorem{prop}[]{Proposition}
\newtheorem{lem}[]{Lemma}

\newtheorem{cor}{Corollary}

\newtheorem{rem}[]{Remark}
\newcommand{\Rset}{\mathbb{R}}

\begin{document}
January 27, 2023.

\begin{center}
\textbf{Some definite integrals arising from selfdecomposable characteristic functions.}

\medskip
Zbigniew J. Jurek

\medskip
University of Wroc\l aw, Wroc\l aw, Poland;

 \ \ $zjjurek@math.uni.wroc.pl$ \ \mbox{and} \ \ $www.math.uni.wroc.pl/^{\sim}zjjurek$
\end{center}

\medskip
\medskip
\emph{Abstract.} In the probability theory \emph{selfdecomposable, or class $L_0$ distributions} play an important role as they are limiting distributions of normalized partial sums of sequences of independent, not necessarily identically distributed, random variables. The class $L_0$ is quite large and includes many known classical  distributions and statistics. For this note the most important feature of the selfdecomposable variables are their random integral representation with respect to L\'evy process. From those random integral representation we get equality of logarithms of some characteristic functions.  These allows us to get formulas for some definite integrals, some of them probably were unknown before.



\medskip
\medskip
The class $L$ of \emph{selfdecomposable distributions} appears in the classical probability theory as a class of limiting distributions obtained from infinitesimal triangular arrays arising from sequences of independent random variables. In the case of independent and identically distributed variables we get the class $\mathbb{S}$ of \emph{stable distributions}, which contains the normal distributions. On the other hand, when one considers  arbitrary infinitesimal triangular arrays we end-up with the class $ID$ of \emph{infinitely divisible distributions}. Thus we have the inclusions $\mathbb{S}\subsetneq L \subsetneq ID.$

Infinitely divisible random  variables $X\in ID$ are uniquely determined by their characteristic functions
$$
\phi_X(t):=\mathbb{E}[\exp(itX)]= \exp\big(ita - \frac{1}{2}\sigma^2t^2+\int_{\Rset\setminus{\{(0)\}}} (e^{itx}-1-\frac{itx}{1+x^2})M_X(dx) \big),
$$
where
$ a\in \Rset, \ \sigma^2\ge 0 \ (\mbox{Gaussian part}), \ \int_{\Rset\setminus{\{(0)\}}}\min(1,x^2)M(dx)<\infty.$

\medskip
The measure $M$ is called \emph{the L\'evy measure} of $X$.  The above integral formula is refered to as  \emph{ the L\'evy-Khintchine representation.}
The triple $a,\sigma^2$ and $M$ is uniquely determined by $X$.

\noindent The classic references to those topics are: Feller (1966), Gnedenko and Kolomogorov (1954) or Loeve (1963).

For infinite divisibility for Banach space valued variables we refer to Araujo and Gine (1980), Chapter 3.

The class $L$ is quite large and includes many classical distributions in probability and statistics; for examples cf. Jurek (1997) and (2021). Moreover, during  last decades the selfdecompsability appeared  in mathematical finance (Carr, Geman , Madan and Yor (2007), a book by Schoutens (2003)) and in statistical physics ( deConinck (1984), deConinck and Jurek (2000), and Jurek (2001)).

\medskip
The aim of this note, (after the preliminary Proposition 1, in the Introduction),  is to show how some definite integrals can be computed from the selfdecomposable characteristic functions; cf. Lemmas 1 , 2 and Corollary 1 \, (for the hyperbolic characteristic functions)  and Corollaries 2-6, in Sections 2 and 3 \, (for the characteristic function expressed via Euler's Euler's gamma function; in particular, for the  Meixner or Feller-Spitzer distributions).

\medskip
\medskip
\textbf{Introduction.}

 A random variable $X$ is called \emph{selfdecomoposable}(in symbols: $X\in L_0$), if
\begin{equation}
\forall\,(t>0)\,\exists \, (X_t \in ID \ \mbox{independent of X}) \ \ X\stackrel{d}{=} e^{-t}X+X_t.
\end{equation}
And inductively, for $k\ge 1$, we define, decreasing sequence of \emph{the Urbanik classes} $L_k$,  as follows:
$$
X\in L_k\ \mbox{iff, in (1),} \ X_t\in L_{k-1}, \ \mbox{for all $t>0$; cf. Urbanik (1972),(1973);} \  (1a)
$$

The "remainders" $(X_t,t>0)$ in (1) satisfy the following\emph{ cocycle equation} $X_{t+s}\stackrel{d}{=}e^{-t}X_s+X_t$. It allows to construct a cadlag L\'evy process and   to infer that
the selfdecomposability of $X$ is equivalent to the claim that there exits unique, in distribution, a L\'evy process  $(Y_X(t),t\ge 0)$, such that
\begin{equation}
X\in L_0 \ \mbox{iff} \ \ X\stackrel{d}{=}\int_0^\infty e^{-t}dY_X(t), \ \  \mbox{and} \ \ \mathbb{E}[\log(1+|Y_X(1)|)]<\infty,
\end{equation}
cf. Jurek and Vervaat (1983),Theorem 3.2, pp. 252-253 or Jurek and Mason (1993), Theorem 3.6.8, pp. 124-126  and Jurek (1982) for a Banach space valued random variables.

To the random variable $Y_X(1)$ we refer to as \emph{the background driving random variable of $X$}; is short: BDRV.

\medskip
\newpage
For this note crucial are the following facts listed in :
\begin{prop}
(a) \ Let  $X$ and $Y_X(1)$ be a selfdecopmosable variable and its BDRV, respectively, in the random integral representation (2). Then
for the characteristic functions $\phi_X(t):=\mathbb{E}[\exp(itX)]$ and $\psi_X(t):=\mathbb{E}[\exp(itY_X(1))]$ we have
\begin{equation}
\psi_X(t)= \exp(t (\log\phi_X(t))^\prime), \ t \neq 0, \ \psi_X(0)=1.
\end{equation}
(b) \ Equivalently, if $[a,\sigma^2,M_X]$ is the triple in the L\'evy-Khintchine formula for $X$ and similarly $[b,s^2,N_X]$ is the triple for $Y_X(1)$ then
\begin{multline}
M_X(B)=\int_{0}^\infty N_X(e^{t}B)dt, \   \mbox{ Borel $B \subset\Rset\setminus{\{(0)\}}$}; \  M_X(\Rset\setminus{\{(0)\}} )= \infty; \\ M_X(dx)=k_X(x)dx, \ \
N_X(dx)= h_X(x)dx, \ \  h_X(x)=(-xk_X(x))^\prime\ge0,  \\   s^2=2\sigma^2,  \ \ \ b=a+\int_{\Rset\setminus{\{((0)\}}}\frac{x}{1+x^2}-\arctan(x))h_X(x)dx.
\end{multline}
\emph{(c)\ [The condition $(-xk_X(x))^\prime\ge 0$ is equivalent to the condition that the function  $xk_X(x)$ is not increasing of both half-lines. This means that
\newline $[a,\sigma^2,k_X(x)]$ is selfdecomposable variable.]}

(d) Hence, by comparing  (3) and $[b,s^2,N_X]$ in (4), we conclude the following identity for $t \in \Rset$,
\begin{multline}
itb -\frac{1}{2}s^2t^2+ \int_{\Rset\setminus{\{((0)\}}}(e^{itx}-1-\frac{itx}{1+x^{2}})h_X(x)dx  =  t(\log\phi_X(t))^\prime .
\end{multline}
\end{prop}

For the proof of (3) cf.  Jurek (2001b), Proposition 3.  For proof of  (4) cf. Jurek and Mason (1993), p.120, formulae (3.6.9)-(3.6.11) for $Q=I$, or  Jurek (1996), p.175. Finally for the equivalences  in the square bracket $[...]$, cf. Jurek and Mason (1993), Theorem 3.4.4,  p. 94 or Steutel and Van Haarn (2004), Theorem 6.12, p. 277.

\medskip
The  above equality (5) is  the key identity for all the definite integrals in this note.

\begin{rem}
\emph{(i) \ If selfdecomposable $X$ has finite second moment then in (5) we will use  the kernel  $\exp(itx)-1-itx $; cf.  Bilingsley (1986), Theorem 28.1, p. 384. Moreover, in that case, in (5), $b=a$.}

\emph{(ii)\ From (4) we see that $X\in L_0$ has Gaussian part if and only if its BDRV $Y_X(1)$ has a Gaussian part.}

\emph{(iii) For the real characteristic functions, as in Section 1 below, we get simpler formulas as we can discard the imaginary part.}
\end{rem}

\medskip
\textbf{1. The hyperbolic characteristic function.}

For the information about the  hyperbolic-type  characteristic function we refer to  Pitaman and Yor (2003) or Jurek and Yor (2004).
Also this section may be viewed as a complement to the recent preprint Jurek (2022 a).

\medskip
\emph{a). \underline{Hyperbolic-sine function}.}

\begin{lem}
From the selfdecomposability of the hyperbolic-sine characteristic function $\phi_{\hat{S}}(t)
=\frac{t}{\sinh(t)}$ we get the following  formulas (L\'evy exponents):
\begin{multline*}
(i) \int_0^\infty(\cos(tx)-1)(\pi/2)csch^2(\pi x/2)dx = 1- t \coth(t); \\
(ii) \int_0^\infty (\cos(tx)-1)\frac{\pi}{2}csch^{2}(\frac{\pi x}{2})\,(\pi x\coth(\frac{\pi x}{2})-1)dx = t^2csch^2(t)-t\coth(t); \\
(iii) \int_{0}^\infty(\cos(tx)-1)\frac{\pi}{4} csch^2(\pi x/2)\big[2\pi^2 x^2\coth^{2}(\pi x/2) +  \pi^2 x^2csch^2(\pi x/2)\\ -6\pi x\coth(\pi x/2)+2\big
]dx = 3t^2csch^2(t)-t\coth(t)(2t^2csch^2(t)+1).
\end{multline*}
\end{lem}
\begin{proof}
Since
\begin{equation*}
\phi_{\hat{S}}(t) = \frac{t}{\sinh t} =[0,0,k_{\hat{S}}(x)]\in ID, \  \  \  \
k_{\hat{S}}(x)=\frac{1}{|x|}\,\frac{1}{e^{\pi |x|}-1}, \  x\in\Rset\setminus{\{(0)\}},
\end{equation*}
i.e., in this example in $a=0$ and $\sigma^2=0$ for the random variable $X:=\hat{S}$.
Since
$h_{\hat{S}}(x):= (-xk_{\hat{S}}(x))^\prime=(\pi/4)csch^2(\pi x/2), \  x\in\Rset\setminus{\{(0)\}},$
is positive thus by part (c) of Proposition 1 we get that $\hat{S}\in L_0$. Its BDRV  $Y_{\hat{S}}(1)=[0,0,h_{\hat{S}}(x)]$, as in view of symmetry of $h_{\hat{S}}(x)$ we have $b=0$ and $s^2=0$ in (4) above.
\newline On the other hand, as
$\psi_{\hat{S}}(t)= \exp(t (\log\phi_{\hat{S}}(t))^\prime) = \exp (1- t\coth(t)$,
 from (6) we get the part (i) of Lemma 1.
\newline Since $$g_{\hat{S}}(x):= - (x h_{\hat{S}}(x))^\prime= \frac{\pi}{4}csch^{2}(\pi x/2)\,(\pi x\coth(\pi x/2)-1) $$
is positive and also symmetric we infer by Proposition 1 that $\hat{S}\in L_1$ (Urbanik class). Moreover, $g_{\hat{S}}(x)$
is the density of the L\'evy measure of the BDRV $Y_{\hat{S}}(1)=[0,0, g_{\hat{S}}]$.

On the other hand ,
$$ t ( \log \psi_{\hat{S}}(t) )^\prime= t(1-t\coth(t))^\prime= t^2csch^2(t)-t\coth(t)  $$
which together with (5) gives the part (ii) of Lemma 1.

Next, using Wolframalpha, we have that
\begin{multline}
r_{\hat{S}}(x):= - (x g_{\hat{S}}(x))^\prime  = \frac{\pi}{8} csch^2(\pi x/2)\\
[2\pi^2 x^2\coth^{2}(\pi x/2) +\pi^2  x^2csch^2(\pi x/2) -6\pi x\coth(\pi x/2)+2]\ge 0,
\end{multline}
is non-negative function therefore, by the part (c) of Proposition 1, we get that $\hat{S}\in L_2$, (Urbanik class).

Finally, since
\begin{equation*}
\exp(t (t^2csch^2(t)-t\coth(t))^\prime)= \exp(3t^2csch^2(t)-t\coth(t)(2t^2csch^2(t)+1))
\end{equation*}
from (5) we get the part (iii) in Lemma 1.
\end{proof}

\medskip
\medskip
\textbf{NOTE 1:} In Jurek (2022a), Theorem 1(a) is proved that $\hat{S}\in L_2\setminus L_3.$ So $\hat{S}\in L_2$
and that fact was used in the proof of Lemma 1.

\medskip
\medskip
\emph{b). \underline{Hyperbolic-cosine function}.}

Similarly as in the section a) for the hyperbolic-sine functions, we have 

\begin{lem}
From the selfdecomposability of the hyperbolic-cosine characteristic function $\phi_{\hat{C}}=\frac{1}{\cosh(t)}$ we get the following L\'evy exponents:
\begin{multline*}
(i) \int_0^\infty(\cos(tx)-1) \frac{\pi}{2}\frac{\cosh(\pi x/2)}{\sinh^2(\pi x/2)}dx= - t\tanh(t);\\
(ii)\int_0^\infty(\cos(tx)-1)\frac{\pi}{4}csch(\frac{\pi x}{2})[\pi x\coth^2(\frac{\pi x}{2})-2\coth(\frac{\pi x}{2})+\pi x csch^2(\frac{\pi x}{2})]dx \\    = -t\tanh(t)-t^2sech^2(t);\\
(iii)\int_0^\infty (\cos(tx)-1)\frac{\pi}{8}csch(\pi x/2)[(\pi x)^2\coth^{3}(\pi x/2) \\ +\coth(\pi x/2)(5(\pi x)^2csch^{2}(\pi x/2)+4)  -
6\pi x\coth^2(\pi x/2)-6\pi x csch^2(\pi x/2)]dx \\ = -t \tanh(t)-t^2(3-2t \tanh t) sech^2(t).
\end{multline*}
\end{lem}
\begin{proof}
For the hyperbolic-cosine variable $\hat{C}$ we have its characteristic function
\begin{equation*}
\phi_{\hat{C}}(t) = \frac{1}{\cosh t} \in L_0;  \
k_{\hat{C}}(x) = \frac{e^{-\pi|x|/2}}{|x|(1-e^{-\pi|x|})} = \frac{1}{2|x|\sinh(\pi|x|/2)};
\end{equation*}
that is, $\hat{C}=[0,0,k_{\hat{C}}]\in L_0$, (by the condition (c) in Proposition 1), with L\'evy spectral measure $M_{\hat{C}}(dx):=k_{\hat{C}}(x)dx$; cf. Jurek (1996), Pitman-Yor (2003) or  Jurek - Yor (2004).

Hence for $h_{\hat{C}}(x):= (-xk_{\hat{C}}(x))^\prime = \frac{\pi}{4}\frac{\cosh(\pi x/2)}{\sinh^2(\pi x/2)}$ is non-negative and symmetric density of L\'evy measure of $Y_{\hat{C}}(1)$ in  (2). Thus in (4) we get $b=0$, that is, $Y_{\hat{C}}(1)=[0,0,h_{\hat{C}}(x)]$.
\newline On the other hand, by (3), and (5)
\begin{multline*}
\psi_{\hat{C}}(t):= \exp(t (\log(\phi_{\hat{C}}(t))^\prime)= \exp(-t \tanh(t))=  [0,0,h_{\hat{C}}(x)]\\ = 
\exp\int_0^\infty (\cos(tx)-1)(  \frac{\pi}{2}\frac{\cosh(\pi x/2)}{\sinh^2(\pi x/2)} )dx,
\end{multline*}
which gives (i) of Lemma 2.

As $Y_{\hat{C}}(1)\in L_0$ (is selfdecomposable) we can repeat the  procedure from the previous step. Since, by WolframAlpha,
\begin{multline*}
g_{\hat{C}}(x):=  (-xh_{\hat{C}}(x))^\prime = \frac{\pi}{8}csch(\frac{\pi x}{2})[\pi x\coth^2(\frac{\pi x}{2}) \\ -2\coth(\frac{\pi x}{2})+\pi x csch^2(\frac{\pi x}{2})],
\end{multline*}
is  a positive and symmetric density of a L\'evy measure of BDRV of $Y_{\hat{C}}(1)$ in (2).
\newline On the other hand, by(3),
$(t (-t \tanh(t))^\prime))= - t \tanh(t)-t^2 sech^2(t))$, which together with the formula for $g_{\hat{C}}(x)$
 proves the equality (ii) in Lemma 2.

Again, as above,  by WolframAlpha,
\begin{multline*}
(-xg_{\hat{C}}(x))^\prime= \frac{\pi}{8}csch(\pi x/2)[(\pi x)^2\coth^{3}(\pi x/2) +\coth(\pi x/2)\\(5(\pi x)^2csch^{2}(\pi x/2)+4)
-6\pi x\coth^2(\pi x/2)-6\pi x csch^2(\pi x/2)],
\end{multline*}
is positive and symmetric density of L\'evy measure we have that $[0,0,g_{\hat{C}}(x)]\in L_0$, by Proposition 1 c). On the other one, by (3),
\begin{multline*}
-t(t\tanh(t)+t^2 sech^2(t))^\prime = -t \tanh(t)-t^2(3-2t \tanh t) sech^2(t),
\end{multline*}
which, in view of (5), gives the identity (iii) in Lemma 2.
\end{proof}

\medskip
\medskip
\textbf{NOTE 2:} In Jurek (2022a), Theorem 1(a) is proved that $\hat{C}\in L_2\setminus L_3.$ So $\hat{C}\in L_2$
and that fact was  used in the proof of Lemma 1.

\medskip
\emph{c). \underline{Hyperbolic-tangent function}.}

For the  hyperbolic tangent variable $\hat{T}=[0,0,k_T(x)$  where the characteristic function and L\'evy measure density are:
$$
\phi_{\hat{T}}(t)=\tanh(t)/ \ \ \mbox{and} \ \ \  k_{\hat{T}}(x)=\frac{1}{2|x|}\frac{e^{-\pi |x|/4}}{\cosh(\pi |x|/4)}.
$$
respectively. By (e) in Proposition 1 we get that $\hat{T}\in L_0$. Then its BDRV $Y_T(1)=[0,0,h_T(x)]$, where
\begin{multline*}
h_{\hat{T}}(t)=(-xk_{\hat{T}})^\prime=\frac{\pi}{8}\frac{1}{\cosh^2(\pi x/4)};  \\  \psi_{\hat{T}}(t)=\exp(t(\log\phi_{\hat{T}}(t))^\prime) =\exp\big[\frac{2t}{\sinh(2t)}-1\big],
\end{multline*}
cf.  fromula (2) in Proposition 1 (a).

Hence we get
\begin{cor}
(i) From the selfdecomposability of hyperbolic tangent $\hat{T}$ we get equality
$$\int_{\Rset\setminus{\{(0)\}}}(cos(tx)-1)\frac{\pi}{8}\frac{1}{\cosh^2(\pi x/4)}dx= \frac{2t}{\sinh(2t)}-1; \ t \in \Rset.$$

(ii) The characteristic function $\psi_T(t)$ represents a compound Poisson distribution therefore $T\notin L_1$ Urbanik class.
\end{cor}
For the part (i), also cf. Jurek-Yor (2004), Proposition 1 with Corollary 1 and the  equality (10). Since L\'evy measures of class $L_0$
are infinite by (4), we get the part (ii) of corollary.

\medskip
\begin{rem}
\emph{}
\emph{
Formulas Lemma 1(i), Lemma 2(i)  and the above  had already appeared in Jurek-Yor (2004). They are added here for the completeness of this presentation.}
\end{rem}

\medskip
\medskip
\textbf{2. Characteristic functions expressed via the Euler's gamma function.}

\emph{(a). \underline{The log-gamma distribution.}}

Log-gamma variables are just the logarithms of the gamma $\gamma_{\alpha,\lambda}$ variables with the parameters $\lambda>0$ (scale) and  $\alpha>0$ (shape).
They are selfdecomposable   
with characteristic functions
\begin{multline}
\phi_{\log \gamma_{\alpha,\lambda}}(t)= e^{-it \log(\lambda)}\, \frac{\Gamma(\alpha+it)}{\Gamma(\alpha)}\\=
\exp[it(\Psi^{(0)}(\alpha)-\log \lambda)+\int_{-\infty}^{0}(e^{itx}-1-itx)\frac{e^{\alpha x}}{|x|(1-e^x)}dx], \\
k_{\log\gamma_{\alpha,1}}(x):=\frac{e^{\alpha x}}{|x|(1-e^x)}1_{(-\infty, 0)}(x),
\end{multline}
where  $ \Psi^{(0)}(z):=d\log \Gamma(z)/dz $ denotes \emph{the digamma function};  cf. Jurek (1997), p. 98 or Jurek (2022), p. 110 ( a comment below the formula (16)). And because $\log\gamma_{\alpha,\lambda}$ have finite second moment ( cf. Corollary 2, in Jurek (2022)) the kernel under the integrand is from  Kolmogorov's formula; cf. Remark 1 (i).

For the random variable $Y_{\log\gamma_{\alpha,1}}(1)$, in the random integral representation (2), we have
$$
h_{\log\gamma{\alpha,1}}(x)= (-x k_{\log\gamma_{\alpha,1}}(x))^\prime=\frac{e^{\alpha x}(\alpha (1-e^{x}) + e^{x})}{(1-e^x)^2}1_{(-\infty,0)}(x),
$$
$b=\Psi^{(0)}(\alpha)-\log \lambda$  and $s^2=0$. Hence
\begin{multline*}
\psi_{\log \gamma_{\alpha,\lambda}}(t)=\exp[it(\Psi^{(0)}(\alpha)-\log \lambda)+\int_{-\infty}^{0}(e^{itx}-1-itx)\frac{e^{\alpha x}(\alpha (1-e^{x})+e^{x})}{(1-e^x)^2}]dx.
\end{multline*}

On the other hand, from (3),  we get
$$
\psi_{\log\gamma_{\alpha,\lambda}}(t)=\exp (t (\log\phi_{\log\gamma_{\alpha,\lambda}}(t))^\prime= \exp(it (\Psi^{(0)}(\alpha +it)-\log\lambda).
$$
All in all, from the identity (5) (taking the Kolmogorov's kernel)  we infer the identity

\begin{cor}
From the selfdecomposability property of the log-gamma variables, for $\alpha >0, \beta>0 $ and $t\in\Rset$ we have

(1) $\int_{-\infty}^{0}(e^{itx}-1-itx)e^{\alpha x}\,\frac{\alpha(1-e^x)+e^{x}}{(1-e^x)^2}]dx = it(\Psi^{(0)}(\alpha +it)-\Psi^{(0)}(\alpha)).$

(2) $\int_{0}^\infty (e^{-itx}-1 +itx)e^{-\beta x}\,\frac{\beta(1-e^{-x}) +e^{-x}}{(1-e^{-x})^2}]dx= it(\Psi^{(0)}(\beta -it)-\Psi^{(0)}(\beta)).$
\end{cor}

\medskip
\medskip
\emph{(b). \underline{Logistic distribution.}}

Let $\emph{l}_\alpha$ denote the logistic distribution; cf. Ushakov (1999), p. 298, Feller (1966), p.52, Jurek (2021), p. 101.  Then (by (7) we have
\begin{multline}
\phi_{\emph{l}_\alpha}(t)=|\frac{\Gamma(\alpha + i t/\pi)}{\Gamma(\alpha)}|^2 = \phi_{1/\pi \log\gamma_{\alpha,1}}(t)\,  \phi_{1/\pi \log\gamma_{\alpha,1}}(-t) \\ =    \exp \int_{-\infty}^\infty(\cos(tx)-1)k_{l_\alpha}(x) dx; \ \ \mbox{with} \  \   k_{l_\alpha}(x)= \frac{1}{|x|}\frac{e^{-\alpha\pi|x|}}{1-e^{- \pi|x|}};
\end{multline}
and by Proposition 1 (c) we get that $\emph{l}_\alpha\in L_0.$

Furthermore, from
\begin{multline}
h_{\emph{l}_\alpha}(x): =(-x  k_{l\alpha}(x))^\prime \\  =
 \frac{\pi}{4}\frac{1}{\sinh^2(\pi|x|/2)} e^{-(\alpha-1)\pi|x|}\{\alpha(1-e^{-\pi|x|})+e^{-\pi|x|}\}>0.
\end{multline}
we conclude that $[0,0,h_{\emph{l}_\alpha}(x)]$ is the background driving variable $Y_{\emph{l}_\alpha}(1)$ in the integral representation (2).

On the other hand from (3) we get
\begin{multline*}
\psi_{\emph{l}_\alpha}(t)=\exp(t (\log(\phi_{\emph{l}_\alpha}(t)))^\prime)
= \exp(t/\pi (i \Psi^{(0)}(\alpha +it/\pi)- i\Psi^{(0)}(\alpha -it/\pi)))\\=
 \exp(t/\pi [i \Psi^{(0)}(\alpha +it/\pi)+\overline{(i\Psi^{(0)}(\alpha +it/\pi))}])
 = \frac{2t}{\pi} \Re[i\Psi^{(0)}(it/\pi +\alpha)].
\end{multline*}
All in all we get
\begin{cor}
From the selfdecomposability property of the logistic  distribution $\emph{l}_{\alpha}, \alpha>0$,  for  $t\in \Rset$,  we have
\begin{equation*}
\int_0^\infty(\cos(tx)-1) \, \frac{\pi}{2}\, \frac{e^{-(\alpha-1)\pi|x|}(\alpha +(1-\alpha)e^{-\pi|x|})}{\sinh^2(\pi|x|/2)}dx
= \frac{2t}{\pi}\,\Re[i\Psi^{(0)}(it/\pi +\alpha)].
\end{equation*}
\end{cor}

\medskip
\medskip
\emph{c). \underline{The generalized z-distribution.}}

For positive parameters $a, b_1,b_2, d $ and $m \in \Rset$ the \emph{generalized z-distribution} $GZ\equiv GZ(a,b_1,b_2,d,m )$ is given by its characteristic function
$$
\phi_{GZ}(t):=\big(\frac{B(b_1+\frac{i a t}{2\pi}, b_2-\frac{i at}{2\pi})}{B(b_1,b_2)}\big)^{2d}\,e^{imt}=\big( \frac{\Gamma(b_1+\frac{i a t}{2\pi})}{\Gamma(b_1)}\, \frac{\Gamma(b_2-\frac{i a t}{2\pi})}{\Gamma(b_2)}  \big)^{2d}e^{imt},
$$
where $B(z_1,z_2)$ denotes the beta-function; cf. Schoutens (2003) , p. 64, or Ushakov (1999) , p. 309.

(For a particular choices of parameters we get Fisher z-distribution; cf. Jurek (2021), the section 3.14, p.105.)

Let us note that the characteristic function $\phi_{GZ}(t)$  can be expressed via characteristic functions of log-gamma variables. Namely, as we have
\begin{multline}
\phi_{GZ}(t)=
\big(\phi_{\log \gamma_{b_1,1}}(at/(2\pi))\,\phi_{-\log \gamma_{b_2,1}}(at/(2\pi))\big)^{2d} e^{imt} \\
= \big(\phi_{(a/2\pi)\log \gamma_{b_1,1}}(t)\,\phi_{-(a/2\pi)\log \gamma_{b_2,1}}(t)\big)^{2d} e^{imt} \ \ t \in \Rset;
\end{multline}
cf. the section (b), on log-gamma variables, above.


Below we will assume  $ a=2\pi, d=1/2$ and $m=0$ and denote $\tilde{GZ}\equiv GZ(2\pi, b_1,b_2, 1/2,0)$.

\medskip
Hence by (7) and section (a) on the log-gamma variables, the $\tilde{GZ}$ distribution has  L\'evy (spectral) measure $\nu_{\tilde{GZ}}(dx):=k_{\tilde{GZ}}(x)dx$ where the density is of the from
\begin{equation}
k_{\tilde{GZ}}(x)= \frac{e^{b_1 x}}{|x|(1-e^{x})}1_{(-\infty,0)}(x) +  \frac{e^{-b_2 x}}{x(1-e^{-x})}1_{(0, \infty)}(x)].
\end{equation}

Hence,
\begin{multline}
h_{\tilde{GZ}}(x):= (-x k_{\tilde{GZ}}(x))^\prime = [ ( \frac{e^{b_1 x}}{(1-e^{x})})^\prime 1_{(-\infty,0)}(x)  - (\frac{e^{-b_2 x}}{(1-e^{-x})})^\prime
1_{(0, \infty )}(x)] \\
=  e^{b_1 x}\, \frac{b_1(1-e^{x})+e^x}{(1-e^{x})^2}\,1_{(-\infty,0)}(x) \\ + e^{- b_2 x} \frac{b_2(1-e^{-x})+e^{-x}}{(1-e^{-x})^2}\,1_{(0,\infty)}(x)>0;
\end{multline}
is the density of L\'evy measure of the  BDRV $Y_{GZ}(1)$. From Corollary 2, in Jurek (2022), we have that log-gamma variables have finite second moment. Consequently, in Kolmogorov's representation, we have
$$
Y_{GZ}(1)=[\Psi^{(0)}(b_1)+\Psi^{(0)}(b_2),0,h_{\tilde{GZ}}(x)]. 
$$
On the other hand, from (5) we get that
\begin{multline}
\psi_{\tilde{GZ}}(t):=\mathbb{E}[\exp(itY_{\tilde{GZ}}(1)]= \exp[
t (\log\phi_{\tilde{GZ}}(t))^\prime] \\
= \exp[t \big(\log\Gamma(b_1+it)+  \log\Gamma(b_2-it)\big)^\prime ] \\
\exp [it(\Psi^{(0)}(b_1 +it)- \Psi^{(0)}(b_2 -it))) ]
\end{multline}
where $\Psi^{(0)}(z):= d Log\Gamma(z)/dz $ is the digamma function. Hence by (6)  we infer the following:
\begin{cor}
From the selfdecomposability of the generalized z-distribution we have the identity
\begin{multline*}
\int_{\Rset \setminus{\{(0}\}}(e^{itx}-1-itx)(e^{b_1 x}\, \frac{b_1(1-e^{x})+e^x}{(1-e^{x})^2}\,1_{(-\infty,0)}(x) \\ +  e^{- b_2 x}\, \frac{b_2(1-e^{-x})+e^{-x}}{(1-e^{-x})^2}\,1_{(0,\infty)} ) dx \\ = it[(\Psi^{(0)}(b_1 +it)-\Psi^{(0)}(b_1))-  (\Psi^{(0)}(b_2 -it)-\Psi^{(0)}(b_2))],
\end{multline*}
for all  $t\in \Rset.$
\end{cor}

\medskip
\textbf{3. Meixner and Feller-Spitzer ( or Bessel) distributions.}

\emph{(a). \underline{The Meixner $\mathbb{M}$ distribution.}}

 For the parameters
$ a>0, \ -\pi<b<\pi, \ d>0, \ m \in \Rset,$ \emph{the probability density function} $f(x;a,b,m,d)$, $ x\in\Rset$,
 \begin{multline}
f(x;a,b,m,d):= \frac{(2\cos(b/2))^{2d}}{2a\pi\Gamma(2d)}\exp(\frac{b(x-m)}{a}) |\Gamma(d+\frac{i(x-m)}{a})|^2, 
\end{multline}
is called Meixner distribution. It has all moments finite and  the characteristic functions is
\begin{equation}
\phi_{\mathbb{M}}(t)\equiv\phi_{\mathbb{M}(a,b,d,m)}(t)=(\frac{\cos(b/2)}{\cosh(\frac{at-ib}{2})})^{2d}\exp(imt);
\end{equation}
cf. Schoutens (2003), p. 62-63.
From above we infer that $\mathbb{M}$ are infinitely divisible. Furthermore, for our purposes, without the loss of generality we assume that
$m=0$ and $d=1/2$.

Being infinitely divisible, with finite second moment,  Meixner distributions admit the following  representations
\begin{multline}
\phi_{\mathbb{M}}(t)= \exp(it\gamma +\int_{-\infty}^{\infty}(e^{itx}-1- itx\,1_{(|x|\le 1)}(x)k_M(dx))), \ \mbox{where} \\
\gamma:=\frac{a}{2}\tan(b/2)-\int_{1}^\infty \frac{\sinh(bx/a)}{\sinh(\pi x/a)}dx,
k_{\mathbb{M}}(dx):=\frac{e^{bx/a}}{2\,x\sinh(\pi x/a)}dx;
\end{multline}
cf. Schoutens (2003), p. 153  for the shift parameter $\gamma$ and p. 154 for the density $k_{\mathbb{M}}(x)$ of L\'evy measure.

This can be extended to Kolmogorov's  type kernel as follows:
\begin{multline*}
\phi_{\mathbb{M}}(t)= \exp(it \tilde{\gamma}+ \int_{-\infty}^{\infty}(e^{itx}-1-itx)k_M(dx))), \ \mbox{where} \\
\tilde{\gamma}:= \gamma + \int_{-\infty}^{\infty}(1-1_{|x|<1}(x))\frac{e^{bx/a}}{2 \sinh(\pi x/a)}dx  = (a/2)\tan(b/2) \\ - \int_{1}^{\infty}\frac{e^{bx/a}-e^{-bx/a}}{2\sinh(\pi x/a)}dx 
 +\int_{(|x|\ge 1)}\frac{e^{bx/a}}{2 \sinh(\pi x/a)}dx  =a/2\tan(b/2), 
\end{multline*}
i.e., $\tilde{\gamma}=\frac{a}{2}\tan(b/2)$ and $\mathbb{M}=[a/2\tan(b/2),0,k_{\mathbb{M}}(x)]$ in Kolmogorov's representation.
Hence, in the random integral representation (2), we have
\begin{multline*}
Y_{\mathbb{M}}(1)=[\frac{a}{2}\tan(b/2), 0,h_{\mathbb{M}}(x)], \\
 h_{\mathbb{M}}(x)= (-xk_{\mathbb{M}}(x))^\prime =\frac{1}{4a}e^{bx/a}\frac{[e^{\pi x/a}(\pi -b)+e^{-\pi x/a}(b+\pi)]}{\sinh^{2}(\pi x/a)},
\end{multline*}
is  positive as $|b|<\pi$. Hence we infer that Meixner  $\mathbb{M}\in L_0$.

On the other hand, from (5)
\begin{multline*}
\mathbb{E}[\exp(itY_{\mathbb{M}}(1))] =  \psi_{\mathbb{M}}(t):= \exp[t (\log\phi_M(t))^\prime] = \exp[-t (\log \cosh((at-ib)/2))^\prime] \\  =\exp[-  at/2\tanh((at-ib)/2)].
\end{multline*}

All in all we get
\begin{cor}
From the selfdecomposability property of Meixner  $\mathbb{M}(a,b,1/2,0)$ variable with constants   $a>0, |b|<\pi$ we have the identity
\begin{multline*}
\int_{\Rset\setminus{\{(0)\}}}(e^{itx}-1-itx)\frac{1}{4a}e^{bx/a}\frac{e^{\pi x/a}(\pi-b)+e^{-\pi x/a}(b+\pi)}{\sinh^2(\pi x/a)}dx \\
= -i\frac{at}{2}\tan(b/2) -\frac{at}{2}\tanh(\frac{at-ib}{2})=  -i\frac{at}{2}\tan(b/2)- \frac{at}{2}\frac{\sinh(at)-i\sin(b)}{\cosh(at)+\cos(b)}\\
= - \frac{at}{2}\frac{\sinh(at)}{\cosh(at)+\cos(b)} - i \,\frac{at}{2}(\tan(b/2) +\frac{\sin(b)}{\cosh(at)+\cos(b)}).
 \end{multline*}
\end{cor}

\begin{rem}

\emph{As an addition to the above :}

\emph{(i) $\tanh[(at-ib)/2]=\frac{at}{2}\frac{\sinh(at)-i\sin(b)}{\cosh(at)+\cos(b)}, \ \mbox{for real}\, a, b, t.$}

\emph{(ii) In particular, for the real and imaginary parts we have
\begin{multline*}
\int_{\Rset}(\cos(tx)-1)\frac{1}{4a}e^{bx/a}\frac{e^{\pi x/a}(\pi -b)+e^{-\pi x/a}(\pi+b)}{\sinh^2(\pi x/a)}dx 
=    - \frac{at}{2}\frac{\sinh(at)}{\cosh(at)+\cos(b)},
\end{multline*}
and
\begin{multline*}
\int_{\Rset}(\sin(tx)-tx))\frac{1}{4a}e^{bx/a}\frac{e^{\pi x/a}(\pi -b)+e^{-\pi x/a}(\pi+b)}{\sinh^2(\pi x/a)}dx  
\\ =  - \,\frac{at}{2}(\tan(b/2)+ \frac{\sin(b)}{\cosh(at)+\cos(b)}).
\end{multline*}
}
\end{rem}

\medskip
\medskip
\emph{(b). \underline{The Feller-Spitzer $FS$ distribution.}}

For $\nu>0$  and the modified Bessel function $I_{\nu}(x)$ we define  the probability density function $p_{\nu}(x):=e^{-x}\frac{\nu I_{\nu}(x)}{x}, \ 0<x<\infty$, which has
\emph{the Feller-Spitzer characteristic function}
\begin{equation}
\phi_{FS(\nu)}(t):=\int_0^\infty e^{itx}p_{\nu}(x)dx=[1-it-\sqrt{(1-it)^2-1}]^{\nu}, t \in \Rset,
\end{equation}
 cf. Feller  p. 414 and p.476 or Ushakov p.283; it is called there
 \emph{Bessel distribution}. For further generalizations of $FS$ distributions cf. Vinogradov and Paris (2021).

 From above we get that $FS$ variable is in ID class and it has L\'evy-Khintchine representation
 \begin{multline}
 \phi_{FS(\nu)}(t)=\exp\nu[ita+\int_0^\infty(e^{itx}-1-\frac{itx}{1+x^2})k_{FS}(x)dx ], \\
\mbox{where} \ \  k_{FS}(x):= \frac{e^{-x}I_0(x)}{x}1_{(0,\infty)}(x);   \    a:=\int_0^\infty \frac{e^{-x} I_{0}(x)}{1+x^2}dx
 \end{multline}
 cf. Jurek (2021) p. 103 and $FS=[a,0,k_{FS}] \in L_0$. Note that the above L\'evy measure density $k_{FS}(x)$
 coincides with the density $\tau_1(x)$, in Vinogradov and Paris (2021), Definition 2..

 \medskip
 For the BDRV $Y_{FS}(1)=[b,0,h_{FS}(x)]$ we get that
 $$ h_{FS}(x):= (-xk_{FS}(x))^\prime = (-e^{-x}I_0(x))^\prime = e^{-x}(I_{0}(x)-I_1(x))>0,$$
 is positive (by Jones (1968) or Paris and Vinogradov (2021) and the inequality (110)) and integrable to 1. For  the shift parameter, by (4), we have
 \begin{multline}
b=a+\int_\Rset(\frac{x}{1+x^2}-\arctan(x))h_{FS}(x)dx \\=
\int_0^\infty \frac{e^{-x}I_0(x)}{1+x^2}dx+\int_0^\infty(\frac{x}{1+x^2}-\arctan(x))e^{-x}(I_0(x)-I_1(x))dx \\
=    \int_0^\infty e^{-x}I_0(x)(\arctan(x))^{\prime}dx - \int_0^\infty \arctan(x))(-e^{-x}I_0(x))^{\prime}dx  \\ +     \int_0^\infty\frac{x}{1+x^2}e^{-x}(I_0(x)-I_1(x))dx = e^{-x}I_0(x)\arctan(x)|_{x=0}^{x=\infty}\\+\int_0^\infty\frac{x}{1+x^2}e^{-x}(I_0(x)-I_1(x))dx
=\int_0^\infty\frac{x}{1+x^2}e^{-x}(I_0(x)-I_1(x))dx,
 \end{multline}
 because $\lim_{x\to 0}e^{-x}I_0(x)\arctan(x)=\lim_{x\to \infty}e^{-x}I_0(x)\arctan(x)=0$.

Consequently, from above and (4) we have
\begin{multline*}
\psi_{FS}(t)=\exp( it (\int_0^\infty\frac{x}{1+x^2}e^{-x}(I_0(x)-I_1(x))dx)\\ +\int_0^\infty(e^{itx}-1- it \frac{x}{1+x^2})e^{-x}(I_0(x)-I_1(x))dx)\\=
\int_0^\infty(e^{itx}-1)e^{-x}(I_0(x)-I_1(x))dx,
\end{multline*}
which is the characteristic function of the compound Poisson distribution, thus $\psi_{FS}(t)\notin L_0,$ and $FS\in L_0\setminus L_1$.

On the other hand, by (5) we have that
 \begin{multline*}
 \psi_{FS}(t)=\exp ( t (\log(\phi_{FS}(t)))^\prime )\\ = \exp(t (\log(1-it-\sqrt{(1-it)^2-1}))^\prime)  = \exp(\frac{it}{\sqrt{-t(t+2i)}}
  \end{multline*}
 which gives
 \begin{cor}
 From the selfdecomposability property of the Feller-Spitzer distribution we have the identity
 $$
 \int_0^\infty(e^{itx}-1)\mu(dx) =\frac{it}{\sqrt{-t(t+2i)}}, \ t\in \Rset.
 $$
 where $\mu(dx):=e^{-x}(I_0(x)-I_1(x))1_{(0,\infty)}(x)dx$ is a probability measure. Both the right and the left expressions are logarithms of infinitely divisible characteristic functions (Poisson compound distributions) .
 \end{cor}

\medskip
\medskip
\textbf{References.}

\medskip
A. Araujo and E. Gine (1980), \emph{The central limit theorem for real and Banach space valued random variables},
John Wiley $\&$ Sons, New York.

\medskip
 P.Billingsley (1986), \emph{Probability and measure}, Second Edition, J.Wiley $\&$ Sons, New York

\medskip
 P. Carr , H. Geman, D. Madan and M. Yor (2007), Self-decomposability and option pricing, \emph{Math. Finance},  vol. 17, No 1 , pp.31-57.

\medskip
 J. deConinck (1984), Infinitely divisible distribution functions of class L and the Lee-Yang Theorem,
Comm. Math. Phys. vol. 96, pp. 373-385.

\medskip
J. deConinck and Z,J. Jurek (2000), Lee-Yang models, selfdecomposability and negative definite functions.
In: High dimensional probability II, E.Gine, D. M. Mason, J. A. Wellner Editors; \emph{Progress in Probab.}
vol.47, Birkhauser 2000, pp. 349-367.

\medskip
 W. Feller (1966), \emph{ An introduction to probability theory and its applications}, vol. II, New York,  J.Wiley$\&$ Sons.

\medskip
 B. V. Gnedenko and A. N. Kolomogorov (1954), \emph{Limit distributions for sums of independent random variables}, Addison-Wesley.

\medskip
B. Grigelionis (1999), Processes of Meixner type, \emph{Lithuanian Math. Journal}, 39(1), pp. 33-41.

\medskip
B. Grigelionis (2000), Generalized z-distributions and related stochastic processes,
\emph{Mathematikos Ir Informaticos Institutas Preprintas}, Nr 2000-22. Vilnius.

\medskip
A. Jones (1968), \emph{An extension of an inequality involving modified Bessel functions},
\emph{J. Math. Phys.} 47, pp.220-221.

\medskip
 Z. J. Jurek (1982), An integral representation of operator-selfdecomposable random variables, \emph{Bull. Acad. Polon. Sci.}
   vol. 30, pp.385-393.

\medskip
Z. J. Jurek (1983), The classes$L_m(Q)$  of probability measures on Banach spaces, \emph{Bull. Acad. Polon. Sci.}, vol. 13, pp. 578-604.

\medskip
Z. J. Jurek (1996), Series of independent random variables, Proc. $7^th$ Japan-Russia Symposium, Tokyo 26-30 July 1995;
S. Watanabe, M. Fukushima, Yu.V. Prohorov, and A. N. Shiryaev Editors; World Scientific, pp. 174-182.

\medskip
Z. J. Jurek (1997),  Selfdecomposability: an exception or a rule? \emph{Annales Uni. M. Curie-Skłodowska},
Lublin -Polonia, vol. LI.1,10, Sectio A, pp. 93-107.

\medskip
Z. J. Jurek (2001a), 1-D Ising models, geometric random sums and selfdecomposability, \emph{Reports on Math. Physics} , vol. 47, pp.21-30.

\medskip
 Z. J.  Jurek (2001b), Remarks on the selfdecomposability and new examples, \emph{Demonstratio Math.} ,vol. XXXIV, no 2, pp.241-250.

\medskip
 Z. J. Jurek (2021), On background driving distribution functions  (BDDF) for some selfdecomposable variables,
\emph{Mathematica Applicanda} , vol 49(2), p.85-109.

\medskip
Z. J. Jurek (2022), Background driving distribution functions and series representations
 for log-gamma self-decomposable random variables, \emph{Theor. Probab. Appl.} vol.67, no 1, pp. 105-117.

\medskip
Z. J. Jurek (2023), Which Urbanik class $L_k$, do the hyperbolic and the generalized logistic characteristic functions belong to?,   math. arXiv:2211.17064v1 [math.PR] 30Nov 2022.

\medskip
 Z. J. Jurek and J.D. Mason (1993), \emph{Operator-limit distributions in probability theory},  John Wiley $\&$ Sons, Inc., New York

\medskip
 Z.J. Jurek and W. Vervaat (1983),  An integral representation for selfdecomposable Banach space valued random variables,
\emph{Z. Wahrscheinlichkeitstheorie und verw. Gebiete}, vol.62, pp.51-62.

\medskip
 Z. J. Jurek and Yor (2004), Selfdecomposable laws associated with hyperbolic functions,
\emph{Probab. Math. Stat.} vol.24 Fasc. 1, pp. 181-190.

\medskip
 M. Loeve (1963), \emph{ Probability theory}, 3rd Edition, D.Van Nostrand Company, Inc. Princeton.

\medskip
 J. Pitman and M. Yor (2003), Infinitely divisible laws associated with hyperbolic functions,
\emph{Canad. J. Math.} vol. 55 (2), pp. 292-330.

\medskip
 W. Schoutens, (2003), \emph{L\'evy processes in finance. Pricing financial derivatives,}  J. Wiley and Sons, England.

\medskip
 F. Steutel and K. van Harn (2004), \emph{Infinite divisibility of probability distributions on the real line }, Marcel Dekker Inc.
 \emph{ J. Wiley and Sons, England}.

 \medskip
 K. Urbanik (1972), Slowly varying sequences of random variables, \emph{Bull. de L'Acad. Polon. Sciences;  Ser. Math. Astr. Phys.}
 28:2, pp. 679-682.

\medskip
K. Urbanik (1973), Limit laws for sequences of normed sums satisfying some stability conditions; Proc. 3rd Internstional Symp.  Multivar. Analysis,
Wright State University, Dayton Ohio, USA, June 19-24, 1972;  Academic Press 1973.

[ Also on: www.math.uni.wroc.pl/~zjjurek/urb-limitLawsOhio1973.pdf]

\medskip
 V. V. Vinogradov  and R. B. Paris (2021), On two extensions of the canonical Feller-Spitzer distribution,
\emph{ J. Stat. Distributions and Appl.}, (2021)8:3, https://doi.org/101186/s40488-021-0013-4

\end{document}